\documentclass[12pt]{amsart}
\usepackage{amsmath, amsthm, amscd, amssymb, graphicx}

\newtheorem{thm}{Theorem}[section]
\newtheorem*{thm*}{Theorem}
\newtheorem{cor}[thm]{Corollary}
\newtheorem*{cor*}{Corollary}

\newtheorem{prop}[thm]{Proposition}
\newtheorem*{prop*}{Proposition}

\newtheorem*{con*}{Conjecture}

\theoremstyle{definition}
\newtheorem{defn}[thm]{Definition}
\newtheorem*{defn*}{Definition}
\newtheorem*{ack}{Acknowledgements}

\newtheorem*{claim}{Claim}

\theoremstyle{remark}

\newtheorem*{example*}{Example}

\newcommand{\R}{\mathbb{R}}
\newcommand{\Z}{\mathbb{Z}}

\begin{document}
\bibliographystyle{plain}

\title{Higher dimensional distortion of random complexes}

\author[Dominic Dotterrer]{Dominic Dotterrer}
\address{Department of Mathematics, University of Chicago} \email{d.dotterrer@math.uchicago.edu}

\begin{abstract}
Using the random complexes of Linial and Meshulam \cite{LM06}, we exhibit a large family of simplicial complexes for which, whenever affinely embedded into Euclidean space, the filling areas of simplicial cycles is greatly distorted.  This phenomenon can be regarded as a higher order analogue of the metric distortion of embeddings of random graphs.
\end{abstract}

\maketitle

\section{Introduction}

One of the natural questions to ask when we come across a new geometric object is ``How does it compare to Euclidean space?"  We examine objects from this viewpoint not only because we all live in Euclidean space, but also because being a subset of Euclidean space is a stringent geometric condition (it implies, in particular, embedability into {\it any} infinite dimensional Banach space \cite{dvoretzky1960}, \cite{milman1971}, \cite{MS2002}).  The study of how well discrete and continuous objects ``fit" into Euclidean space, Banach spaces, or geometric space forms is extensive (the literature is vast but the author might suggest \cite{LLR1995}, or \cite{matousek2002}, as a place to start).  Finding good embeddings (or obstructions to them) for discrete structures (graphs, groups, finite metric spaces, etc.) into $L^2$ or $L^1$ is an industry in its own right, particularly because of the direct applications to theoretical computer science.

\subsection{Metric distortion}

We begin by giving a seminal example of the quantitative study of obstructions to ``nice" embeddings:

\begin{thm*}[J. Bourgain, \cite{bourgain1985}]

\quad

Denote by $G_{n,p}$ the Erd\H{o}s--R\'enyi random graph on $n$ vertices, i.e. the probability space of graphs on $n$ vertices such that, for a given graph, $G$, $$\mathbb{P} [G] = p^{E}(1-p)^{\binom{n}{2} - E}$$ where $E$ is the number of edges in $G$. \\

There is a $K>0$ such that if $p = \frac{K \log n}{n}$, then with probability tending to $1$ as $n \to \infty$, any embedding $\phi : G \to \mathcal{H}$ of $G \in G_{n,p}$ into a Hilbert space must have $$\max_{x,y \in G} \frac{\Vert \phi (x) - \phi (y) \Vert}{d(x,y)} \cdot \max_{x,y \in G}\frac{d(x,y)}{\Vert \phi (x) - \phi (y) \Vert} \geq  C \frac{\log n}{\log \log n}.$$

Here, $\Vert \cdot \Vert$ is the Euclidean norm in of the Hilbert space $\mathcal{H}$, and $d(\cdot, \cdot)$ is the graph metric in $G$, i.e. $d(x,y)$ is the length of the minimal path from $x$ to $y$ in $G$.
\end{thm*}

\bigskip

The left side of the inequality is referred to as the {\it metric distortion} of $\phi$ (we will denote it by $\delta_0(\phi)$).  Metric distortion of embeddings is one of many avenues in which random constructions, such as Erd\H{o}s--R\'enyi random graphs, have been of great value to geometry.  It is somewhat paradoxical that random (generic) objects can be good examples of extremal geometries.

However, we have no intention of adding to this very lively discussion of the theory of metric embeddings.  Instead we will be most interested in {\it higher order} phenomena, namely the properties of embeddings of random simplicial complexes into Euclidean space.  First, then, we must decide what we mean by ``higher dimensional metric distortion."

\subsection{Filling distortion}

In this paper, we will be seeking obstructions to the existence of embeddings of geometric objects which preserve some higher geometric structure.  With this in mind we propose the following definition.

\begin{defn}\label{distortion}

\quad

Let $X$ be a simplicial complex and $\phi: X \to \mathcal{H}$ a Lipschitz map from $X$ to a Hilbert space.  The {\it filling distortion} of $\phi$ is given by:
$$\delta_1 (\phi) = \sup_{z \in B_1 X} \frac{{\rm Fill}_\mathcal{H}\,  \phi_*z}{{\rm Fill}_X\,  z} \cdot \sup_{z \in B_1 X} \frac{{\rm Fill}_X \, z}{{\rm Fill}_\mathcal{H} \, \phi_* z}.$$

Here, $B_1 X \subset C_1 X$ refers to the set of $\Z_2$ cellular boundaries of dimension $1$ in $X$.  For such a $z \in B_1 X$,  ${\rm Fill}_X (z) = \min \{ \Vert y \Vert \; : \; y \in C_2 X, \; \partial y = z \}$, where $\lVert y \rVert = {\rm supp} (y)$ is the number of non-zero coefficients of $y$. Analogously, for a $\Z_2$ Lipschitz, $1$-cycle in $\mathcal{H}$, ${\rm Fill}_\mathcal{H} (z) = \inf \{ {\rm vol}_2 (y) \; : \; y \text{    a Lipschitz   } 2\text{-chain, and    } \partial y = z\}$.
\end{defn}

Let us pause to make a few points about this definition.

\begin{enumerate}
\item $\Z_2$ 0-boundaries are just collections of points, where the number of points in each path component is even.  For a $0$-boundary, $z$,  $Fill (z)$ simply means the shortest cumulative length of geodesics connecting pairs of these points.  Said another way, if we associate the two points $x$ and $y$ with the obvious $0$-boundary, ${\rm Fill}(z) = d(x,y)$ (where $d(\cdot, \cdot)$ means geodesic distance).  With this in mind, it becomes clear that our definition is indeed a higher dimensional analogue of metric distortion.\\

\item A map $\phi : X \to \mathcal{H}$ induces a $\Z_2$-linear map $\phi_* : B_1 X \to B_1 \mathcal{H}$.  Both $B_1 X$ and $B_1 \mathcal{H}$ are naturally endowed with a {\it flat metric} \cite{whitney}:
    $$d_\flat (z,w) := {\rm Fill} (z-w) = \inf \{ \Vert y \Vert \;  : \; \partial y = z-w\}.$$
    Considering $B_1 X$ and $B_1 \mathcal{H}$ as metric spaces, we see that the {\it filling distortion} of $\phi$ is exactly the {\it metric distortion} of $\phi_*$, or $$\delta_1 (X \overset{\phi}{\to} \mathcal{H} ) = \delta_0 (B_1 X \overset{\phi_*}{\to} B_1 \mathcal{H} ). $$  This is the main motivation for defining filling distortion as we did.

\end{enumerate}

With a definition in place, we can then look for a candidate simplicial complex to achieve a high level of distortion.  Our main theorem addresses this:

\begin{thm}\label{main-intro}

\quad

For every large $n$  and for $\epsilon >0$, there exists a $2$-dimensional simplicial complex on $n$ vertices, with complete $1$-skeleton ($\binom{n}{2}$ edges), with the property that any affine map $\phi: X \to \mathcal{H}$ into a Hilbert space must have $$\delta_1 (\phi) \geq C n^{\frac{1-2\epsilon}{4}}.$$

\end{thm}
Although we have chosen the definition to complement the notion of metric distortion, we would like to take a moment to contrast this estimate with its lower dimensional analogue.  A seminal result in the theory of metric embeddings is the early theorem of Bourgain:

\begin{thm}[J. Bourgain, \cite{bourgain1985}]
There are constants, $C$ and $K$, such that for every finite metric space $(X,d)$, there is a map $\phi : X \to \R^{K \log \vert X \vert}$ with $$\delta_0 (\phi) \leq C \log \vert X \vert$$ where $\vert X \vert$ is the number of points in $X$.

\end{thm}

This theorem says that every finite metric space can be embedded in Euclidean space with only {\it logarithmic} metric distortion.  Contrasting with Theorem \ref{main-intro}, we see that higher dimensional distortion, despite its analogy to metric distortion, exhibits wholly different geometric behavior; higher dimensional distortion can be as large as a {\it power} of the number of vertices (or faces).

For an obvious (and not very judicious) upper bound, one can consider the map, $X \to \R^N$, which takes each vertex of $X$ to one of the standard basis elements of $\R^N$.  For such a map, $$ \max_{z \in B_1 X} \frac{ {\rm Fill}_\mathcal{H} (\phi_* z)}{{\rm Fill}_X (z)} = O(1) \qquad and \qquad \max_{z \in B_1 X} \frac{ {\rm Fill}_X (z)}{{\rm Fill}_\mathcal{H} (\phi_* z)} \leq \vert X^{(2)} \vert.$$

As frivolous as this bound may seem, there is no good evidence that there does not exist complexes all of whose affine embeddings have filling distortion greater than $c \lvert X^{(2)} \rvert^{1-\epsilon}$, and in fact, this may be plausible.\\

In the course of proving our main theorem, we will prove an intermediate proposition of independent interest.  The proposition serves to relate the filling distortion of a map $X \to \mathcal{H}$ to geometric information on $X$, namely, the spectral gap of the cellular Laplacian and a measure of the sparsity of $X$.

\begin{prop}\label{main-prop}

\quad

Let $X$ be a $2$-dimensional simplicial complex on $n$ vertices, a complete $1$-skeleton and the smallest eigenvalue of the Laplacian acting on $1$-forms given by $\lambda^1(X)$. Then any affine map $\phi : X \to \mathcal{H}$, suitably scaled so that $${\rm Fill}_\mathcal{H} (\phi \tau ) \geq {\rm Fill}_X (\tau)$$ for every $1$ cycle, $\tau$, of length $3$, must have $$\sum_{f \in X^{(2)}}  \big( {\rm Area} (\phi f ) \big)^2 \geq \frac{\lambda^1(X)}{3(n-2)} \sum_\tau \big( {\rm Fill}_X \tau \big)^2.$$
where the first sum runs over $2$-dimensional faces in $X$ and the second sum runs over all $1$-cycles of length $3$ in $X$.

\end{prop}

Here, $\lambda^1 (X)$ refers to the spectral gap of the (up-down) Laplacian acting on cellular $1$-cochains, $\Delta = d^* d : C^1 (X; \R) \to C^1 (X; \R)$ (we will elaborate on this further in the next section).\\

With this inequality in mind, we will attempt to maximize the quantity, $$\frac{\lambda^1(X)}{\vert X^{(2)} \vert } \sum_\tau \big( {\rm Fill}_X \tau \big)^2.$$

To get a better sense of what this quantity measures, we can consider the analogy in graphs: $$ \frac{\lambda^0 (G)}{\vert E \vert} \sum_{x,y \in G} d_G (x,y).$$ This quantity is asymptotically optimized on families of $3$-regular Ramanujan graphs \cite{lub}.  Schematically, $\lambda^0 (G)$ quantitatively measures the connectivity of $G$, while $\sum_{x,y} d_G(x,y)$ measures the edge sparsity of $G$.\\

Just as Bourgain did with random graphs, we rely on random simplicial complexes to find examples of complexes with lower bounds on this quantity.

\begin{defn*}[Linial and Meshulam, \cite{LM06}]

\quad

Let $Y_{n,p}$ denote the {\it Linial--Meshulam random complex}, the probability space of $2$-dimensional simplicial complexes on $n$ vertices, with complete $1$-skeleton (i.e. $\Delta^{(1)}_n \subset Y \subset \Delta^{(2)}_n$), such that for a given such $Y$, $$\mathbb{P} [ Y ] = p^F (1-p)^{\binom{n}{3} - F} \qquad \text{   where } F \text{ is the number of faces in } Y.$$

\end{defn*}

When, $p = K \log n$ for a sufficiently large $K$, the random complexes of Linial and Meshulam end up giving the estimate in theorem \ref{main-intro} with high probability as $n$ tends to infinity.

\subsection{the $\ell^1$-volumes of Newman and Rabinovich}

After the first version of this article was written, the author became aware of the work of Newman and Rabinovich in \cite{NR2010}.  In brief, this work constitutes an extension of the more classical considerations of metric embeddings to the realm of {\it finite volume spaces}, which include simplicial complexes with complete lower skeleta as an important example.  In their article, Newman and Rabinovich prove:

\begin{thm*}[Newman and Rabinovich, \cite{NR2010}]
For every large $n$, there exists a two dimensional complex, $X$, with $n$ vertices and a complete $1$-skeleton so that every affine embedding $\phi: X \to \mathcal{H}$ from $X$ to a Hilbert space has, $$\delta_1 (\phi) \geq c n^{\frac{1}{5}}.$$
\end{thm*}

In order to achieve this theorem, they also use random complexes and a different version of proposition \ref{main-prop}.

\begin{prop*}[\cite{NR2010}]

For a map, $\phi$, which satisfies the hypotheses of Proposition \ref{main-prop},
$$\delta_1 (\phi) \geq  c(X) \sum_\tau {\rm Fill}_X (\tau)$$
\end{prop*}
where $c(X)$ is defined in the following way.  Consider the cochain complex of $\Z_2$-vector spaces: $$ 0 \to C^{-1} X \overset{d}{\longrightarrow}  C^{0} X \overset{d}{\longrightarrow} C^{1} X \overset{d}{\longrightarrow} C^{2} X \to 0$$
each endowed with the Hamming ($L^1$) norm, i.e. $\lVert y \rVert$ is the number of non-zero coefficients of $y$.  For each $\alpha \in C^1 X = C^1 \Delta_n$ (since the $1$-skeleton is complete), we have both $\Vert d \alpha \Vert_\Delta$ and $\Vert d \alpha \Vert_X$ (clearly, $\Vert d\alpha \Vert_X \leq \Vert d \alpha \Vert_\Delta$) and $c(X)$ is defined as:
$$c(X) =\frac{1}{\vert X^{(2)}\vert} \min_{\alpha \in C^1 X} \frac{\Vert  d \alpha \Vert_X}{\Vert d \alpha \Vert_\Delta}.$$

With this proposition in mind, we emphasize that our approach, in particular proposition \ref{main-prop}, is simultaneously complementary and distinct in that it provides an explicit connection to higher spectral information rather than the sparsest cut.  These two notions have enjoyed an obverse relationship in the world of graphs for sometime \cite{lub}.  Their relationship in higher dimensions is coming to be understood \cite{simplicial-isoperimetric}.  In general, however, spectral gap information is sometimes strictly weaker than sparsest cut information (see theorem 1.2 in \cite{simplicial-isoperimetric}).  It is therefore a bit surprising that we obtain a slightly better bound via spectral methods rather than through sparsest cut techniques.

It is worth noting that the definition of Newman-Rabinovich of {\it finite volume spaces} only coincides with with our definition in the case of a simplicial complex with complete lower skeleta.  Each definition is a distinct generalization, ours generalizing to simplicial complexes in general and theirs generalizing to a general class of weighted $k$-hypergraphs.

\subsection{Other results on maps to Euclidean space}

There have already been some recent results on maps from random simplicial complexes to Euclidean space of a topological, rather than explicitly geometric, nature.  We have two in particular in mind.  The first is Gromov's point selection theorem for random complexes:

\begin{thm*}[Gromov, \cite{sing2} (combined with the main observation in \cite{DK11})]

\quad

There exists a constant, $c$, such that if $p > \frac{K \log n}{n}$ (for a big enough $K$) then with probability tending to $1$ as $n$ tends to infinity, $Y \in Y_{n,p}$ has the property that for any continuous map $\phi : Y \to \R^2$, there exists a point $p \in \R^2$ which lies in the image of at least $c \vert Y^{(2)} \vert$ of the $2$-dimensional faces of $Y$.

\end{thm*}

This theorem simply says that a Linial--Meshulam random complex, once it has enough faces, will be forced to pile up, or ``overlap" (\cite{overlap}) whenever mapped into $\R^2$.\\

Another recent result addresses the question of topological embeddability:

\begin{thm*}[Wagner, \cite{wagner2011}]

\quad

If $p > \frac{K}{n}$, then with probability tending to $1$ as $n$ tends to infinity, $Y \in Y_{n,p}$ does not topologically embed in $\R^4$

\end{thm*}

This is the higher dimensional analogue of the fact that Erd\H{o}s--R\'enyi random graphs are overwhelmingly non-planar (for a large enough $p$).\\

The reader should be informed that we did not state either of these theorems in nearly their highest generality, but instead offered them in this form to lend them better to the theme of this article.

\subsection{Overview}
In the next section we will develop our notation and be more explicit with our definitions.  Once our notation is in place we will prove proposition \ref{main-prop} in section \ref{dilation}.  We will follow up in section \ref{estimates} by estimating the desired spectral and isoperimetric quantities of Linial--Meshulam random complexes, thereby obtaining theorem \ref{main-intro}.  We reserve a section to describe how each of these results generalize to higher dimensions.  The final section will address a series of questions and remarks.

\begin{ack}
First and foremost, I must acknowledge Larry Guth, for his reassuring interest in this project.  Over a casual lunch, Larry suggested to me some version of definition \ref{distortion} and I am indebted for that suggestion.  Also, in the spring of 2011, I spent several afternoons with Larry and Matt Kahle discussing systolic and isoperimetric inequalities in random complexes; proposition \ref{fill} is a direct product of those discussions.  In addition, I am grateful to Alfredo Hubard for those long walks to and from the Dinky in which we discussed several incarnations of this problem.  Finally, I must thank Uli Wagner for drawing my attention to the work of Newman and Rabinovich mentioned above.

\end{ack}

\section{Notation and concepts}

Throughout, we will concern ourselves with a simplicial complex, $X$. We will denote the set of $k$-dimensional faces of $X$ as $X^{(k)}$.  This induces two chain complexes of interest:

$$ 0 \to C^{-1}_\R X \overset{d}{\longrightarrow}  C^{0}_\R X \overset{d}{\longrightarrow} C^{1}_\R X \overset{d}{\longrightarrow} C^{2}_\R X \to 0$$ where $C^k_\R X =\{ X^{(k)} \to \R \}$ and for $\beta \in C^{k-1}_\R X$, $d\beta (y) = \sum_{x \in \partial y} \beta (x)$.\\

For real cochains, the norm is understood to be $\Vert \beta \Vert = \sqrt {\sum_{x \in X^{(k)}} \vert \beta (x) \vert^2}$.\\
\begin{defn}
For a simplicial complex, $X$, we define the (up-down) $k$-th spectral gap to be $$\lambda^k (X) = \inf_{\beta \in C^k_\R X} \frac{\Vert d\beta \Vert^2 }{\inf_{\alpha \in C^{k-1}_\R X } \Vert \beta + d\alpha \Vert^2} = \inf_{\beta, \, \partial \beta = 0} \frac{\Vert d \beta \Vert^2}{\Vert \beta \Vert^2}$$
where both infimums run over non-zero chains.
\end{defn}
(see \cite{hanlon}, \cite{forman}, \cite{simplicial-isoperimetric} for more on the combinatorial Hodge decomposition, and see \cite{DK11} for more on coboundary expansion).\\

We will also consider the chain complex with $\Z_2$-coefficients: $$0 \leftarrow C_{-1} X \overset{\partial}{\longleftarrow}C_{0} X \overset{\partial}{\longleftarrow}C_{1} X \overset{\partial}{\longleftarrow}C_{2} X \overset{\partial}{\longleftarrow} 0.$$

This chain complex will be endowed with the $L^1$ (or {\it Hamming}) norm, $\Vert y \Vert = \vert {\rm supp} y \vert$.\\

We will denote the space of $k$-cycles by $Z_k X = \ker \partial_k \subset C_k X$.  The space of cycles has a natural metric, the {\it flat metric} \cite{whitney}: $$d_\flat (z,w) := {\rm Fill}_X (z-w) = \inf \{ \Vert y \Vert : \partial y = z -w \}.$$

\section{Dilation estimates for embeddings}\label{dilation}

The section will be devoted to proving our main proposition:

\begin{prop*}

\quad
Let $\Delta^{(1)}_n \subset X \subset \Delta^{(2)}_n$ be a $2$-dimensional simplicial complex with complete $1$-skeleton.  Let $\phi : X \to \mathcal{H}$ be an affine embedding of $X$ into an infinite dimensional Hilbert space, suitably scaled so that $${\rm Fill}_\mathcal{H} (\psi \tau ) \geq {\rm Fill}_X (\tau)$$ for every triangle, $\tau$, then,$$\sum_{f \in X^{(2)}}  \big( {\rm Area} (\phi f ) \big)^2 \geq \frac{\lambda^1 (X)}{3(n-2)} \sum_\tau \big( {\rm Fill}_X \tau \big)^2.$$
where the first sum runs over $2$-dimensional faces in $X$ and the second sum runs over all triangles in $X$.
\end{prop*}

\begin{proof}
We can assume that the image of the $0$-skeleton, $X^{(0)}$ forms a linearly independent set in $\mathcal{H}$ (since a small perturbation of the vertices does not change the areas of triangles very much).

Choose orthonormal coordinates, $x_1, \dots, x_n$ for ${\rm span} X^{(0)} \cong \mathbb{R}^n$.  We will let $\phi$ induce a function $\psi: X^{(1)} \to \mathbb{R}^{\binom{n}{2}}$ defined by $$\psi_{(i<j)} (e) = \frac{1}{2} \int_{\phi (e)} x_i  dx_j - x_j dx_i + \sum^n_{m=0}\int_{\phi(e)} y^{(i<j)}_{m} dx_m \quad \text{     for each    } e \in X^{(1)}$$ (with the $y^{(i<j)}_{m}$ as fixed constants to be chosen later).

Now we have $d\psi : X^{(2)} \to \mathbb{R}^{\binom{n}{2}}$, and by the Stokes theorem applied to a given face $f$, $$(d\psi (f))_{(i<j)} = \int_{\phi f} dx_i \wedge dx_j.$$  Now it is easily seen that $\vert d\psi (f) \vert^2 = \big( {\rm Area} (\phi f) \big)^2$.  This is because the area form of $\phi f$ can we written as $$\omega_{\phi f} = \sum_{i < j} a_{(i<j)} dx_i \wedge dx_j \: \text{        where        }\: \sum_{i<j} a^2_{(i<j)} = 1 \: \text{       and           }\: \int_{\phi f} dx_i \wedge dx_j = a_{(i<j)} {\rm Area} (\phi f).$$
 Thus, $$\Vert d\psi \Vert^2 = \sum_{f \in X^{(2)}}  \big( {\rm Area} (\phi f ) \big)^2.$$

We will need to prove a small claim:

\begin{claim}
Consider the function, $\xi : X^{(1)} \to \R^{\binom{n}{2}}$ given by $$(\xi (e))_{(i<j)} = \frac{1}{2} \int_{\phi (e)} x_i dx_j - y_j dx_i. $$ Then for any $\alpha: X^{(0)} \to \R^{\binom{n}{2}}$, we can choose $y^{(i<j)}_{m}$ so that $\psi = \xi + d\alpha$.
\end{claim}

\begin{proof}
If $e = [v,w] \in X^{(1)}$,

$$ \sum_m \int_{\phi(e)} y^{(i<j)}_{m} dx_m = \langle y^{(i<j)} , \phi (v) \rangle - \langle y^{(i<j)} , \phi (w) \rangle$$

Since $\phi (X^{(0)})$ is linearly independent, for every function $f: X^{(0)} \to \R$, there exists a corresponding $y \in \R^n$ such that $$f(x) \equiv \langle y , \phi (x) \rangle \quad \text{   for every   } x \in X^{(0)}.$$

Therefore, for every function, $f: X^{(0)} \to \R^{\binom{n}{2}}$ we can choose $\binom{n}{2}$ such $y \in \R^n$ (denoted $y^{(i<j)}$) such that $$(f_{(1<2)} (x) , \dots, f_{(n-1<n)} (x) ) \equiv (\langle y^{(1<2)}, \phi (x) \rangle, \dots, \langle y^{(n-1<n)} , \phi (x) \rangle  \text{   for every   } x \in X^{(0)}.$$
\end{proof}

Since we can choose $(y_1, \dots, y_n)$ so that $\partial \psi = 0$, we have the inequality:

$$ \sum_{f \in X^{(2)}}  \big( {\rm Area} (\phi f ) \big)^2 = \Vert d\psi \Vert^2 \geq \lambda^1 (X) \Vert \psi \Vert^2.$$

Now we have only to prove that $$\Vert \psi \Vert^2 \geq \frac{1}{3(n-2)} \sum_\tau \big( {\rm Fill_X} (\tau) \big)^2.$$

We observe that for a triangle $\tau$ formed by the edges $e_1$, $e_2$, and $e_3$, we have (by the Stokes theorem again), $$\Big[  \psi (e_1) + \psi(e_2) + \psi (e_3) \Big]_{(i<j)} =\int_{\phi \tau} dx_i \wedge dx_j$$

So that, $$\sum_{(i<j)} \Big[  \psi (e_1) + \psi(e_2) + \psi (e_3) \Big]^2_{(i<j)} = {\rm Fill}_\mathcal{H} (\phi \tau) \geq {\rm Fill}_X (\tau)$$

and by Cauchy-Schwartz: $$\vert \psi (e_1) \vert^2 + \vert \psi (e_2) \vert^2 + \vert \psi (e_3) \vert^2 \geq \frac{1}{3} \big( \vert \psi (e_1) \vert + \vert \psi (e_2) \vert + \vert \psi (e_3) \vert \big)^2.$$

Summing over all triangles, each edge is contained in $n-2$ triangles, we have $$(n-2) \Vert \psi \Vert^2 \geq \frac{1}{3} \sum_\tau \big( {\rm Fill}_X (\tau) \big)^2.$$
\end{proof}

In light of this proposition, it should be clear to the reader that we seek to find $2$-dimensional complexes which maximize the quantity, $$\frac{\lambda^1(X)}{\vert X^{(2)} \vert} \sum_\tau \big( {\rm Fill}_X (\tau) \big)^2.$$
As we increase the number of $2$-faces, $\lambda^1 (X)$ will go up, but $\frac{\sum ( {\rm Fill}_X (\tau) )^2}{\vert X^{(2)}\vert}$ will go down.

It is not clear to the author how to build exact optimizers for this quantity, so in the next section we will resort to using random complexes a la Linial and Meshulam \cite{LM06}.\\

\section{Filling estimates for random complexes}\label{estimates}

Since we have given ourselves the liberty to take estimates up to a constant, we will exhibit a somewhat cavalier indifference to preserving sharp quantities.

We will rely on the geometry of random complexes.  We recall the definition:

\begin{defn*}[Linial and Meshulam, \cite{LM06}]

\quad

Let $Y_{n,p}$ denote the {\it Linial--Meshulam random complex}, the probability space of $2$-dimensional simplicial complexes on $n$ vertices, with complete $1$-skeleton (i.e. $\Delta^{(1)}_n \subset Y \subset \Delta^{(2)}_n$, where $\Delta_n$ denotes the complete simplicial complex, the $(n-1)$-dimensional simplex), such that $$\mathbb{P} [ Y \in Y_{n,p} ] = p^F (1-p)^{\binom{n}{3} - F} \qquad \text{   where } F \text{ is the number of faces in } Y.$$

\end{defn*}

\begin{prop}\label{spec}

\quad
Let $\Delta^{(1)}_n \subset Y \subset \Delta^{(2)}_n$ be a $p$-random complex.  There is a constant, $C$, so that if $p \geq \frac{C\log n}{n}$, then, with probability tending to $1$ as $n \to \infty$,  $$\lambda^1 (Y) \geq\frac{1}{3} p n.$$

\end{prop}

The proof of this proposition is a simple consequence of theorem 2 in \cite{gundert2012}:

\begin{thm*}[Gundert and Wagner, \cite{gundert2012}]
Let $\hat{\lambda}^1(X)$ denote the spectral gap of {\it normalized} (up-down) Laplacian on $1$-forms on a simplicial complex, $X$.  For all $c>0$, there exists a constant $K$ such that if $p \geq \frac{K \log n}{n}$ and $Y = Y_{n,p}$ is a random $2$-complex (with complete $1$-skeleton) then $$\hat{\lambda}^1 (Y) \geq 1 - \frac{K}{\sqrt{p n}}$$ with probability greater than $1 - n^{-c}$.
\end{thm*}

Here the {\it normalized Laplacian} can be obtained from the Hodge (up-down) Laplacian, $\partial d$, by normalizing the rows: For each edge, $e$, let ${\rm deg}(e)$ denote the number of $2$-faces of $Y$ that contain $e$.  If we use the indicator functions $\mathbf{1}_e$ as the basis for $C^1 (Y; \R)$ and write the normalized up-down Laplacian, $L$, as a matrix with respect to this basis, then we have $ L = D \partial d $, where $D$ is a diagonal matrix whose diagonal entries corresponding to the basis element $\mathbf{1}_e$ is $\frac{1}{{\rm deg}(e)}$ (see \cite{gundert2012} or references therein for more information).

As a result, we have the following inequality: $$\lambda^1(Y) \geq \hat{\lambda}^1 (Y) \cdot \min_{e \in Y^{(1)}} {\rm deg}(e).$$

\begin{proof}[proof of Proposition \ref{spec}]

In order to prove proposition \ref{spec}, we simply need to prove the following claim:

\begin{claim}
With probability tending to $1$ as $n$ tends to infinity, the degree of each edge is greater than $\frac{p(n-2)}{2}$.
\end{claim}
\begin{proof}

Our argument is a standard one.

The expected degree of each edge is $p (n-2)$.  By a form of Chernoff's inequality \cite{chernoff}, each edge, $e$, has $$\mathbb{P} [{\rm deg}(e) < (1-\epsilon)p(n-2) ] \leq e^{- \frac{\epsilon^2 p (n-2)}{2}}.$$
Taking $\epsilon = \frac{1}{2}$, and taking a union bound:
$$\begin{array}{lll}
\sum_e \mathbb{P} \Big[{\rm deg}(e) &<& \frac{p(n-2)}{2} \Big]\\
 &\leq& e^{\frac{p (n-2)}{8}}\\
 &\leq& \binom{n}{2} \cdot n^{- \frac{K}{8}}  \to 0 \\
 && \text{  by taking    } K > 16.
\end{array}$$

\end{proof}

Now applying this to the theorem of Gundert and Wagner we have $$\lambda^1 (Y) \geq \frac{p(n-2)}{2} - K\sqrt{pn} \geq \frac{pn}{3} \qquad \text{for large  }n.$$

\end{proof}

Now we are left to find an lower bound on $\sum_\tau \big( {\rm Fill}_Y (\tau) \big)^2.$

\begin{prop}\label{fill}
Let $\Delta^{(1)}_n \subset Y \subset \Delta^{(2)}_n$ be a $p$-random complex with $p = n^{\epsilon-1}$, then, with probability tending to $1$ as $n \to \infty$, $$\sum_\tau \big( {\rm Fill}_Y (\tau) \big)^2 \geq Cn^{\frac{5}{2}-\epsilon}$$
where the sum runs over all cycles, $\tau$, of length $3$.
\end{prop}

\begin{proof}
First, shall examine a single cycle, $\tau$, of length $3$ and bound the probability that ${\rm Fill}(\tau) < n^\alpha$.  To achieve this, we appeal to an estimate made in \cite{ALLM10} (later revised to \cite{ALLM2010}), but attributed as an observation of Eran Nevo, that a $k$-cycle $z \in Z_k \Delta_n$ which does not contain any smaller cycles as a subset and which is supported on $f_0 (z)$ vertices and $f_d (z)$ faces of dimension $d$ must have $$f_0 \leq \frac{f_d +  (d+2)(d-1)}{d}.$$

It is important to note that in dimension $2$, this inequality is simply saying that a {\it minimal} cycle (i.e. one that contains no other cycles as a strict subset) must have Euler characteristic less than $2$.  In the final section, we will make remarks regarding analogous results in higher dimensions where the above formula will be of use.

Now, taking a {\it minimal} filling of $\tau$ (i.e. a chain, $y$, such that $\partial y = \tau$ and $y$ contains no other fillings of $\tau$ as a strict subset) we can obtain a minimal cycle in $\Delta_n$ by including $y$ and the triangle in $\Delta_n$ which bounds $\tau$.  Then, using the inequality above, the number of fillings of $\tau$ of size $m$ in $\Delta_n$ can be bounded by,

$$\begin{array}{lll}
\binom{n}{f_0 - d-1} \binom{\binom{f_0}{d+1}}{m} &\leq& n^{f_0 -d-1} \Big( \frac{e f^{d+1}_0}{m} \Big)^m \\

& \leq  &n^{\frac{m+(d+2)(d-1) - d^2 - d}{d}} (C m^d)^m \\

& =& n^{-\frac{2}{d}} (Cn^{\frac{1}{d}}m^d)^m

\end{array}$$

Therefore, setting $d=2$, we have,
$$\begin{array}{lll}
\mathbb{P} [ \exists y, \, \partial y = \tau, \, \Vert y \Vert < n^\alpha ] &\leq& n^{-1} \sum^{n^\alpha}_{m\geq 3} (C p n^{\frac{1}{2}} m^2)^m \\
&\leq& n^{-1} \sum_{m\geq 3} (en^{2\alpha +\epsilon -\frac{1}{2}})^m
\end{array} $$

So that,
$$\mathbb{P} [ \exists y, \, \partial y = \tau, \, \Vert y \Vert < n^\alpha ]\;  \leq \; C n^{3(2\alpha +\epsilon -\frac{1}{2})-1}\frac{n^{(n^\alpha -2) (2\alpha +\epsilon -\frac{1}{2})}-1}{n^{2\alpha +\epsilon -\frac{1}{2}}-1}$$

Therefore, if we set $2\alpha +\epsilon -\frac{1}{2} <0 $, then we have $$\mathcal{P} [ {\rm Fill}_Y (\tau) < n^\alpha ] \to 0 $$ and $$\mathbb{E}[{\rm Fill}_Y \tau] \geq cn^{\frac{1 -2\epsilon}{4}} \qquad \text{   for large enough   } n.$$\\

 For the next step of the proof, we will bound, from above and below, the quantity $$\mathbb{E} \bigg[ \sum_\tau \min \{ {\rm Fill}_Y (\tau), n^{\frac{1-2\epsilon}{4}} \} \bigg].$$

On the one hand,
$$\begin{array}{lll}
\mathbb{E} \bigg[ \sum_\tau \min \{ {\rm Fill}_Y (\tau), n^{\frac{1-2\epsilon}{4}} \} \bigg] &\geq& \binom{n}{3} n^{\frac{1 - 2\epsilon}{4}} \mathbb{P}[ {\rm Fill}_Y (\tau) \geq n^{\frac{1 - 2\epsilon}{4}} ] \\
&\geq& \frac{99}{100} \binom{n}{3}n^{\frac{1 - 2\epsilon}{4}}.
\end{array}$$

On the other hand, if $Y$ is chosen according to $Y_{n,p}$, and denote by $H$ the event that at least at least $\frac{1}{100}\binom{n}{3}$ cycles, $\tau$, of length $3$ have ${\rm Fill}_Y (\tau) \geq \frac{n^{\frac{1-2\epsilon}{4}}}{99}$.  

Notice that $H$ implies the proposition since, if $H$, then $$\sum_\tau \Bigl( {\rm Fill}_Y (\tau)\Bigr)^2 \geq \frac{1}{100}\binom{n}{3} \cdot \Biggr( \frac{n^{\frac{1-2\epsilon}{4}}}{99} \Biggl)^2$$

Now, if $H$ holds, then certainly $$\mathbb{E} \bigg[ \sum_\tau \min \{ {\rm Fill}_Y (\tau), n^{\frac{1-2\epsilon}{4}} \} \bigg] \geq \binom{n}{3} n^{\frac{1-2\epsilon}{4}}.$$

On the other hand, if $H$ does not hold, then there are less than $\frac{1}{100}\binom{n}{3}$ of cycles of length $3$ with whose filling area is larger than $\frac{n^{\frac{1-2\epsilon}{4}}}{99}$ so that

$$\mathbb{E} \bigg[ \sum_\tau \min \{ {\rm Fill}_Y (\tau), n^{\frac{1-2\epsilon}{4}} \} \bigg] \leq \bigg[ \frac{99}{100}\binom{n}{3}\frac{n^{\frac{1-2\epsilon}{4}}}{99} + \frac{1}{100}\binom{n}{3} n^{\frac{1-2\epsilon}{4}} \bigg]$$

Combining these, we have, $$ \mathbb{E}_Y \bigg[ \sum_\tau \min \{ {\rm Fill}_X (\tau), n^{\frac{1-2\epsilon}{4}} \} \bigg] \leq \binom{n}{3} n^{\frac{1- 2\epsilon}{4}} \mathbb{P}[H] + \bigg[ \frac{99}{100}\binom{n}{3}\frac{n^{\frac{1-2\epsilon}{4}}}{99} + \frac{1}{100}\binom{n}{3} n^{\frac{1-2\epsilon}{4}} \bigg] (1-\mathbb{P} [H]) $$

Putting the upper and lower bounds together,

 $$\binom{n}{3} n^{\frac{1-2\epsilon}{4}} \leq \binom{n}{3} n^{\frac{1- 2\epsilon}{4}} \mathbb{P}[H] + \bigg[ \frac{99}{100}\binom{n}{3}\frac{n^{\frac{1-2\epsilon}{4}}}{99} + \frac{1}{100}\binom{n}{3} n^{\frac{1-2\epsilon}{4}} \bigg] (1-\mathbb{P} [H])$$

which yields,
$$\mathbb{P}[H] + \frac{1}{50} (1 - \mathbb{P}[H]) \geq \frac{99}{100} \qquad \Rightarrow \qquad \mathbb{P}[H] \geq \frac{97}{98}.$$

\end{proof}

\begin{cor}
Let $\Delta^{(1)}_n \subset Y \subset \Delta^{(2)}_n$ be a $p$-random complex with $p = n^{\epsilon -1}$.  Then with probability tending to $1$ as $n \to \infty$, every affine embedding $\phi : X \to \mathcal{H}$ of $X$ into an infinite dimensional Hilbert space, $\mathcal{H}$ must have, $$\max_{z \in B_1 Y} \frac{ {\rm Fill}_\mathcal{H} (\phi_* z)}{ {\rm Fill}_Y (z) } \cdot \max_z  \frac{ {\rm Fill}_Y (z) }{ {\rm Fill}_\mathcal{H} (\phi_* z)} \geq C n^{\frac{1-2\epsilon}{4}}$$
\end{cor}

\begin{proof}
Taking an affine map, $\phi: Y \to \mathcal{H}$ and scaling it so that ${\rm Fill}_\mathcal{H} (\phi_* \tau ) \geq {\rm Fill}_Y (\tau)$ (again, the filling distortion is continuous with respect to small perturbations, so we may always perturb $\phi$ slightly and then scale it as prescribed).  Then there is a $2$-face of $X$ such that: $$\big({\rm Area}(\phi f)\big)^2 \geq \frac{\lambda^1(Y)}{3(n-2)\vert Y^{(2)} \vert} \sum_\tau \big( {\rm Fill}_Y \tau \big)^2 \geq c \frac{1}{n^3} \sum_\tau \big( {\rm Fill}_Y \tau \big)^2 \geq cn^{\frac{1-2\epsilon}{2}}$$
Therefore, $$\delta_1 (\phi) \geq cn^{\frac{1- 2\epsilon}{4}} .$$

\end{proof}

\section{Filling distortion in higher dimensions}

We have chosen to state all of the theorems and propositions in this article in terms of $2$-dimensional complexes.  We felt that writing all arguments in their generality was cumbersome and of little use to the reader.  However, we would be doing the reader a disservice if we mentioned nothing about how the theorems and propositions generalize to higher dimensions.  To this end, we have devoted the current section to a brief sketch of the propositions and proofs of the preceding sections along with annotations describing what minor changes must be made in higher dimensions.

Ultimately, theorem \ref{main-intro} generalizes to:

\begin{thm*}
For every large $n$ and every $\epsilon$, there exists a $(k+1)$-dimensional simplicial complex on $n$ vertices and complete $1$-skeleton with the property that every affine map $\phi : X \to \mathcal{H}$ has, $$\delta_1 (\phi) \geq C n^{\frac{k- (k+1)\epsilon}{(k+1)^2}}$$
\end{thm*}

The three tools needed for the proof are,
\begin{enumerate}
\item\label{sketch1} a generalization of proposition \ref{main-prop}, \\
\item\label{sketch2} an estimate on the spectral gap of the Laplacian acting on the $k$-forms of a random complex \\
\item\label{sketch3} and an estimate on the average filling volume of a $k$-cycle of volume $k+2$ in a random complex.
\end{enumerate}

To obtain the objectives (\ref{sketch2}) and (\ref{sketch3}) we need only observe that, first, the theorem of Gundart and Wagner \cite{gundert2012} is stated for all dimensions, and second, that to get started in the proof of proposition \ref{fill} we needed only that every $d$-dimensional cycle, $z$, which does not include another $d$-cycle as a proper subset must have:
$$f_0(z) \leq \frac{f_d (z) +  (d+2)(d-1)}{d}.$$

Objective \ref{sketch1} requires a bit more careful consideration. However, the most important aspect of the proof of proposition \ref{main-prop} is the construction of a real cochain, $\psi$, with $\partial \psi =0$.

In general, we will use the (vector valued) cochain, $\psi: X^{(k)} \to \R^{\binom{n}{k+1}}$ defined by
$$\begin{array}{lll}
\psi_{(i_0 < \cdots < i_k)} (\sigma) &=& \frac{1}{2}\int_{\phi \sigma} \sum_j (-1)^j x_{i_j} dx_{i_0} \wedge \cdots \wedge \hat{dx_{i_j}} \wedge \cdots \wedge dx_{i_k}\\
&& + \int_{\phi \sigma} \sum_{(j_1 < \cdots < j_k)} y^{(i_0 < \cdots < i_k)}_{(j_1 < \cdots < j_k)} dx_{j_1} \wedge \cdots \wedge dx_{j_k}.
\end{array}$$

(in light of this cumbersome formula, it may occur to the reader now why we decided to omit the general case).  This cochain has the benefit of

\begin{enumerate}
\item $d\psi : X^{(k+1)} \to \R^{\binom{n}{k+1}}$ is given by $$d \psi_{(i_0 < \cdots < i_k)} (\sigma) = \int_{\phi (\sigma) } dx_{i_0} \wedge \cdots \wedge dx_{i_k}.$$\\
\item The constants $y^{(i_0 < \cdots < i_k)}_{(j_1 < \cdots < j_k)}$ can be chosen, via a general position argument just as in the proof of proposition \ref{main-prop} to ensure that $\partial \psi = 0$.
\end{enumerate}

\section{Remarks and questions}

In this section, we give a few remarks and a few open questions.

\subsection{The volume-respecting embeddings of U. Feige}

There is another noteworthy generalization of the concept of metric distortion.  Feige \cite{feige} defined the notion of {\rm volume-respecting embeddings}; let us define it here.

\begin{defn*}[Feige, \cite{feige}]
If $(S,d)$ is a finite metric space, the {\it volume} of $S$ is defined as the supremum of the volume of the convex hull of the image of $S$ under a $1$-Lipschitz map from $S$ to $\R^{\vert S \vert -1}$.  More formally, $${\rm Vol}(S) := \sup_{\phi, \, 1-\text{Lipschitz}} {\rm vol}_{\vert S \vert -1} \big[{\rm convex}(\phi(s_1), \dots, \phi (s_{\vert S \vert}))\big].$$

Now, let $(X,d)$ be a finite metric space endowed with an additional hypergraph structure, $\chi \subset 2^X$ (we can take, for example, all subsets of $X$ of size less than $M$).  Then for a $1$-Lipschitz map $\Phi : X \to \mathcal{H}$ the {\it volume distortion} of $\Phi$ is defined as $$\eta (\Phi) := \max_{S \in \chi} \bigg[ \frac{{\rm Vol} (S)}{{\rm vol}_{|S|-1} {\rm convex}(\Phi (S))}\bigg]^{\frac{1}{\vert S \vert -1}}$$ where ${\rm Vol}(S)$ is the volume of $S$ as a metric space in its own right.

\end{defn*}

\begin{example*}
Let $(X,d)$ be a finite metric space and let $\chi$ be the hypergraph structure consisting of all pairs of points in $X$, then for any map $\phi : X \to \mathcal{H}$, $$\eta (\phi) = \delta_0 (\phi).$$ So Feige's volume distortion is indeed a generalization of metric distortion.
\end{example*}

Volume distortion, however, is distinct from filling distortion as we have defined it.  A simplicial complex simply regarded as a metric space ignores the higher skeleta.  For example, a simplicial complex on $n$ vertices with complete $1$-skeleton (and any hypergraph structure desired) can be embedded by some $\phi$ into $\R^n$ with $\eta (\phi) =1$ by simply sending the vertices to an orthonormal basis.  This shows, in particular, the dependence of volume distortion on the underlying metric.

\subsection{Filling distortion and the fundamental group}

In the last few years there has been some innovative work of the topology of Linial--Meshulam complexes.  The work of Babson, Hoffman and Kahle is a prime example.

\begin{thm*}[Babson, Hoffman, Kahle, \cite{BHK11}]

\quad
For any $\delta >0 $, if $p > n^{-\frac{1}{2}+\delta }$, then with probability tending to $1$ as $n \to \infty$, $Y \in Y_{n,p}$ has $\pi_1 (Y) = 0$,

\quad

and if $p  < n^{-\frac{1}{2} -\delta}$, then asymptotically almost surely, $Y \in Y_{n,p}$ has $\pi_1 (Y) \neq 0$.

\end{thm*}

Now if we reexamine our main theorem, we notice that our estimates on filling distortion break down when $\epsilon =\frac{1}{2}$.  The reason for this is explained in \cite{BHK11}:  The fundamental group of $Y$ vanishes at the same threshold that every triangle, $\tau$, is the boundary of a disk with a {\it bounded} number of faces.  Therefore, our lower bound on $\sum_\tau ({\rm Fill}_Y \tau )^2$ must degenerate at the threshold $p = n^{-\frac{1}{2}}$ and we only obtain $\delta_1 (\phi) \geq c$ for some small constant (which is moot since $\delta_1 (\phi)$ is always greater than $1$).  The degeneration of the lower bound on filling distortion is not simply a failure of our method because, $$\inf_\phi \delta_1 (\phi) \leq \max_\tau {\rm Fill}_X \tau$$ (just by embedding the vertices of the complex as the standard basis of $\R^n$).

\subsection{Manifolds}

If one emulates the proof \cite{matousek2002} that a $k$-regular graph, $G$, (with its shortest distance metric) requires $C(k) \sqrt{\lambda^0} \log \vert G \vert$ metric distortion to embed into Euclidean space, we immediately see that the proof extends to hyperbolic manifolds:

\begin{prop*}
Let $(M^n, g)$ be a hyperbolic manifold whose Laplacian (on functions) has spectral gap $\lambda^0 (M,g)$. Then any map $\phi: M\to \mathcal{H}$ has metric distortion, $$\delta^0 (\phi) \geq C \sqrt{\lambda^0} \log {\rm vol}_n (M,g).$$
\end{prop*}

It seems natural then, to ask if there is an estimate, $$\delta^k (\phi) \geq F( \lambda^k (M,g), {\rm vol}_n (M,g) ),$$ for every map, $\phi$, from a hyperbolic manifold into $\mathcal{H}$.

\subsection{Extremal complexes}

Given the estimate in proposition \ref{main-prop}, it seems a natural question to ask: What simplicial complexes (on $n$ vertices and complete $1$-skeleton) maximize the quantity:

$$\frac{\lambda^1 (X)}{\vert X^{(2)} \vert} \sum_\tau \Big( {\rm Fill}_X (\tau)\Big)^2 \qquad ?$$

The author has no idea how to systematically optimize this quantity, but the optimizers may be very interesting.

\bibliography{embed-complexes}

\end{document}